\documentclass{amsart}
\usepackage{amsmath}
\usepackage{geometry}
\usepackage{amsthm}
\usepackage{amssymb}
\usepackage{graphicx}
\usepackage{hyperref}
\usepackage{enumitem}
\usepackage{color}
\usepackage{comment}
\usepackage{setspace}
\newcommand{\bN}{\mathbb{N}}
\newcommand{\bbN}{\mathbb{N}}

\newcommand{\bZ}{\mathbb{Z}}

\newcommand{\bR}{\mathbb{R}}

\newcommand{\Aut}{\operatorname{Aut}}
\newcommand{\CAT}{\operatorname{CAT}}

\newcommand{\id}{\operatorname{id}}

\newcommand{\COS}{\operatorname{COS}}

\newcounter{ClaimCounter}
\newtheorem{theorem}{Theorem}[section]
\newtheorem{proposition}[theorem]{Proposition}
\newtheorem{lemma}[theorem]{Lemma}
\newtheorem{corollary}[theorem]{Corollary}

\newtheorem{claim}[ClaimCounter]{Claim}
\theoremstyle{definition}
\newtheorem{definition}[theorem]{Definition}
\newtheorem{example}[theorem]{Example}
\theoremstyle{remark}
\newtheorem{rem}[theorem]{Remark}
\newenvironment{remark}{\begin{rem}}{\end{rem}}

\title[Directions in Hyperbolic groups]{The space of directions for hyperbolic totally disconnected locally compact groups}
\author{Timothy P.~Bywaters}
\address{School of Mathematics and statistics\\
	The University of Sydney\\
	NSW 2006 Australia}
\email[T.~P.~Bywaters]{t.bywaters@maths.usyd.edu.au}
\date{October 2019}

\begin{document}

\begin{abstract}
	The space of directions is a notion of boundary associated to an arbitrary totally disconnected locally compact group. We explicitly calculate the space of directions of a group acting vertex transitively with compact open vertex stabilisers on a locally finite connected hyperbolic graph. These are examples of groups where techniques from geometric group theory can be generalised from the discrete to the non-discrete case. We show the space of directions for these groups is a discrete metric space. Our results resolve a conjecture of Baumgartner, M\"{o}ller and Willis in the affirmative.
\end{abstract}
\maketitle
\tableofcontents
\section{Introduction}
The study of totally disconnected locally compact (t.d.l.c.) groups naturally arises from the desire to understand general locally compact groups. Any locally compact group is the extension of a connected group by a t.d.l.c. group. The solution to Hilbert's fifth problem shows that connected locally compact groups can be approximated by Lie groups, see \cite{Yamabe53} which builds on numerous works including \cite{Kuranishi50},\cite{Gleason52}, \cite{Montgomery52} and \cite{Yamabe53B}. As a result, a lot of focus is now on t.d.l.c. groups.

Beyond the existence of a basis of compact open subgroups given in \cite{vDa31}, very little was known about the general structure of t.d.l.c. groups until the publication of \cite{Willis94}. Here the concepts of scale function and tidy subgroup where introduced for an inner automorphism of a t.d.l.c. group. These concepts have been generalised to endomorphisms, see \cite{Willis15}, but have also been built upon to give other invariants, examples include the flat rank \cite{Willis04} and the space of directions \cite{Baumgartner06}. Aside from a general understanding of t.d.l.c. groups, these ideas have had surprising applications in random walks and ergodic theory \cite{Dani06, Jaworski96, Previts03}, arithmetic groups \cite{Shalom13} and Galois theory \cite{Chat12}.

There is hope that the ideas stemming from the scale function and tidy subgroup may contribute to the development of a geometric or combinatorial structure associated to an arbitrary t.d.l.c. group. Something analogous to symmetric spaces for Lie groups or buildings for algebraic groups. To better understand how such a structure should look, it is desirable to have an understanding of the relation between general concepts and specific examples. Our work fits within this framework. 

We focus on groups acting vertex transitively with compact open vertex stabilisers on a locally finite connected $\delta$-hyperbolic graph. Such groups are called hyperbolic. Many results from hyperbolic geometric group theory can be generalised to these non-discrete groups. Theorem \ref{thm:axis} is such a generalisation and is of independent interest. The theorem shows that any hyperbolic automorphism of a locally finite connected $\delta$-hyperbolic graph has some power which acts as a translation along a bi-infinite geodesic. This is reminiscent of the notion of axis for hyperbolic automorphisms of trees. Our main theorem, Theorem \ref{thm:mainthm}, shows that the space of directions of a hyperbolic group is a discrete metric space and is related to the hyperbolic boundary of the underlying graph. Our calculations resolve  \cite[Conjecture 32]{Baumgartner12} in the affirmative.
\subsection{Acknowledgements}
The author is indebted to Udo Baumgartner, Jacqui Ramagge, Stephan Tornier and George Willis for their fruitful discussions, suggestions, stimulating questions and examples. This research was conducted while the author was supported by the Australian Government Research Training Program.
\subsection{Structure and conventions}
The remainder of this document is split into three sections. Section \ref{sec:hyp_prelim} is a preliminary section which sets notation and introduces results that are used in later sections. In particular, the scale function and tidy subgroups, the space of directions, hyperbolic metric spaces and hyperbolic groups are all covered.

Section \ref{sec:InvariantGeodesic} contains a generalisation of a theorem from \cite{Delzant91} to the non-discrete case, see Theorem \ref{thm:axis}. Effectively, we show that certain hyperbolic automorphisms of connected locally finite $\delta$-hyperbolic graphs have some power which acts as a translation. This generalises the original theorem which was restricted to elements of finitely generated hyperbolic groups acting on a Cayley graph.

We prove our main result, Theorem \ref{thm:mainthm}, in Section \ref{sec:hypSpaceOfDirections}. The proof uses the notion of axis from Section \ref{sec:InvariantGeodesic} and hyperbolic geometry to control distances between relevant vertices. These distances are then used to give bounds on the coset indices required in the space of directions pseudometric calculation. We conclude with some corollaries which stem from our main result. 
\subsubsection{Conventions}
Set $\bN = \{1, 2, \ldots\}$ and $\bN_0 = \{0\}\cup \bN$. Topological groups are assumed to be Hausdorff. For a group $G$ acting on a set $X$ and $Y\subset X$, we let $G_{\{Y\}} = \{g\in G\mid g(Y) = Y\}$ and $G_{Y} = \{ g\in G\mid g(y) = y \hbox{ for all }y\in Y\}$. If $X$ is a graph, we say $G$ is acts vertex transitively $G$ acts transitively on the of vertices.
\section{Preliminaries}\label{sec:hyp_prelim}
The theory of t.d.l.c. groups is rich and rapidly developing. Here we outline a small section which is required to understand our results.
\subsection{The scale function and tidy subgroups}
We summarise the basic properties of the scale function and tidy subgroups. These two closely related concepts were introduced in \cite{Willis94} but have been built up in \cite{Willis01} and \cite{Willis15}.

Following \cite{vDa31} which gives a basis of compact open subgroups for an arbitrary t.d.l.c. group, we make the following definition.
\begin{definition}
	Let $G$ be a t.d.l.c. group. Set 
	\[\COS(G) = \{U\le G \mid U \hbox{ is compact and open}\}.\]
\end{definition}
\begin{definition}
	Let $G$ be a t.d.l.c. group and $\alpha\in\Aut(G)$. We define the \emph{scale} of $\alpha$ to be the natural number
	\[s(\alpha):= \min\{[\alpha(U):\alpha(U)\cap U]\mid U\in\COS(G)\}.\]
	Any compact open subgroup which realises the scale is called \emph{tidy}.
	The scale induces a function $s:G\to \bbN$, which we also call the \emph{scale}, by setting the scale of $g\in G$ to be the scale of the inner automorphism corresponding to $g$. An automorphism $\alpha$ is said to be \emph{uniscalar} if $s(\alpha) = 1 = s(\alpha^{-1})$.
\end{definition}
A major achievement of \cite{Willis94} is to classify precisely when a subgroup is tidy. To give the classification, for $G$ a t.d.l.c. group, $\alpha\in \Aut(G)$ and $U\in\COS(G)$, we set 
\begin{alignat*}{3}
&U_+ = \bigcap_{n\ge 0} \alpha^n(U) \hbox{, } &U_{++} = \bigcup_{n\ge 0}\alpha^n(U_+),\\
&U_- = \bigcap_{n\le 0} \alpha^n(U)\quad \hbox{and} \quad&U_{--} = \bigcup_{n\le 0}\alpha^n(U_-).
\end{alignat*}
\begin{theorem}\label{thm:tid_min}
	Let $G$ be a t.d.l.c. group and $\alpha\in \Aut(G)$. A subgroup $U$ is \emph{tidy} for $\alpha$ if and only if
	\begin{enumerate}
		\item $U = U_+U_-$. In this case we say $U$ is \emph{tidy above}.
		\item $U_{--}$ (or equivalently $U_{++}$ if $U$ is already tidy above) is closed. In this case we say $U$ is \emph{tidy below}.
	\end{enumerate}
\end{theorem}
\begin{remark}
	Theorem \ref{thm:tid_min} has been extended to endomorphisms in \cite{Willis15}. In this case, the definitions of $U_+$ and $U_{++}$ need to be altered to apply to endomorphisms. Also, in the setting of endomorphisms, $U_{--}$ being closed is no longer equivalent to $U_{++}$ being closed. 
\end{remark}
The proof of Theorem \ref{thm:tid_min} relies on a process called a tidying procedure. This procedure takes in as input a compact open subgroup and produces a subgroup which is tidy above and below for a given automorphism. Different tidying procedures have been adapted for a variety of purposes, for example see \cite{Willis04} for a tidying procedure which applies to abelian groups of automorphisms and \cite{Moe02} for a tidying procedure which is used to build a representations of t.d.l.c. groups into tree automorphism groups in \cite{Baumgartner04}. See \cite{Bywaters17B} for a generalisation to endomorphisms. 

Theorem \ref{thm:tid_min} and the associated tidying procedure are instrumental in proving the following properties of tidy subgroups and properties of the scale function.

\begin{proposition}[Properties of tidy subgroups]\label{prop:prop_of_tid}
	Suppose $G$ is a t.d.l.c. group with $U\in\COS(G)$ tidy for $\alpha\in\Aut(G)$. Then
	\begin{enumerate}[label = (\roman*)]
		\item $U$ is tidy for $\alpha^n$ for all $n\in\bZ$;
		\item If $V\in\COS(G)$ is tidy for $\alpha$, then $U\cap V$ is also tidy for $\alpha$;
		\item $\alpha^n(U)$ is tidy for $\alpha$ for all $n\in\bZ$.
	\end{enumerate}
\end{proposition}

\begin{proposition}[Properties of the scale function]\label{prop:scale_prop}
	Suppose $G$ is a t.d.l.c. group and $\alpha\in\Aut(G)$. Then
	\begin{enumerate}[label= (\roman*)]
		\item $s(\alpha) = 1$ if and only if there exists $U\in\COS(G)$ such that $\alpha(U)\le U$;
		\item $\alpha$ is uniscalar if and only if there exists $U\in\COS(G)$ such that $\alpha(U) = U = \alpha^{-1}(U)$;
		\item $s(\alpha^n) = s(\alpha)^n$ for all $n\in\bbN_0$;
		\item If $\beta\in\Aut(G)$, then $s(\alpha) = s(\beta\alpha\beta^{-1})$.
	\end{enumerate}
	Furthermore, the scale function on $G$ is continuous and $\Delta(g) = s(g)/s(g^{-1})$, where $g\in G$ and $\Delta$ is the modular function on $G$.
\end{proposition}

\subsection{The space of directions}
We summarise the results of \cite{Baumgartner06}. There, the authors construct a metric space at infinity for an arbitrary t.d.l.c. group. For this section fix a t.d.l.c. group $G$. Note that for $\alpha\in\Aut(G)$, the map $U\in\COS(G)\mapsto \alpha(U)\in \COS(G)$ defines an action of $\Aut(G)$ on $\COS(G)$. With appropriate choice of metric, this is an action by isometries.
\begin{lemma}[{\cite[Section 2]{Baumgartner06}}]
	For $U,V\in \COS(G)$, set
	\[d(U,V) := \log([U:U\cap V][V:U\cap V]).\]		
	Then the pair $(\COS(G),d)$ is a metric space on which $\Aut(G)$ acts by isometries.
\end{lemma}
We can now reinterpret the scale function and tidy subgroups in terms of the action on $\COS(G)$. Part \ref{itm:lem_dir_obv2} if Lemma \ref{lem:basic_cos_action} follows immediately from the definitions whereas part \ref{itm:lem_dir_obv1} is \cite[Lemma 3 Part (i)]{Baumgartner06}.
\begin{lemma}\label{lem:basic_cos_action}
	Suppose $\alpha\in\Aut(G)$ and $U\in\COS(G)$. Then
	\begin{enumerate}[label = (\roman*)]
		\item \label{itm:lem_dir_obv1} The set $\{d(\alpha^n(U),U)\mid n\in\bN_0\}$ is bounded if and only if $s(\alpha) = 1$;
		\item\label{itm:lem_dir_obv2} $U\in \COS(G)$ is tidy for $\alpha$ if and only if $d(\alpha(U), U) = \min_{V\in\COS(G)}d(\alpha(V), V)$. In this case we have $d(\alpha(U),U) = \log(s(\alpha))+\log(s(\alpha^{-1}))$. 
	\end{enumerate}	
\end{lemma}

We use $\COS(G)$ to build a space at infinity for $G$. Like many boundary constructions associated to metric spaces, our construction involves quotienting infinite rays by an asymptotic relation before defining a suitable topology. We start with our notion of rays.

\begin{definition}
	Suppose $\alpha\in \Aut(G)$. Then for any $U\in \COS(G)$, the \emph{ray generated by $\alpha$ based at $U$} is the sequence $(\alpha^{n}(U))_{n\in\bN_0}$.
\end{definition}

Observe that if $(\alpha^n(U))_{n\in\bN_0}$ and $(\alpha^n(V))_{n\in\bN_0}$ are two rays generated by $\alpha$ based at $U,V\in\COS(G)$, then $d(\alpha^n(U),\alpha^n(V)) = d(U,V)$ trivially is bounded. It is desirable to have a definition that is independent of base point. We expand on this notion of bounded. 

\begin{definition}\label{def:asymptotic_rays}
	Two sequences $(V_n)_{n\in\bN_0}$ and $(U_{n})_{n\in\bN_0}$ of compact open subgroups of $G$ are \emph{asymptotic} if $d(V_n,U_n)$ is bounded.
\end{definition}
We use the notion of asymptotic sequences of compact open subgroups to define a notion of asymptotic on automorphisms. We will use rays generated by the automorphisms at a given base point. Naively, we could require that these two rays be asymptotic in the sense of  Definition \ref{def:asymptotic_rays}; however, this would not account for the speed at which an automorphism moves towards infinity. To accommodate for this, a scaling factor of the automorphism is allowed in the definition. This is given by $k_{\alpha}$ and $k_{\beta}$ in Definition \ref{def:asymp}.

\begin{definition}\label{def:asymp}
	Let $\alpha,\beta\in \Aut(G)$. We say $\alpha$ and $\beta$ are \emph{asymptotic} and write $\alpha\asymp \beta$ if there exists $k_\alpha,k_\beta\in\bN$ and $U_{\alpha},U_{\beta}\in\COS(G)$ such that the rays generated by $\alpha^{k_\alpha}$ and $\beta^{k_{\beta}}$ based at $U_{\alpha}$ and $U_{\beta}$ respectively are asymptotic.
\end{definition}

\begin{proposition}[{\cite[Lemma 10]{Baumgartner06}}]
	The relation $\asymp$ is an equivalence relation and is independent of base point subgroups.
\end{proposition}

It is not hard to see from Lemma \ref{lem:basic_cos_action} that all automorphisms with scale $1$ form an asymptotic class. We distinguish these automorphisms from others.

\begin{definition} We say an automorphism $\alpha$ of $G$ \emph{moves towards infinity} if $s(\alpha)> 1$. 
\end{definition}

\begin{lemma}\label{lem:asymp_towards_inf}
	Suppose $\alpha$, $\beta\in \Aut(G)$ are asymptotic with $\alpha$ moving towards infinity. Then $\beta$ moves towards infinity.
\end{lemma}

\begin{definition}
	For $A\subset \Aut(G)$, let $A_>$ be the subset of automorphisms moving towards infinity. For $\alpha\in A_>$, let $\partial_A(\alpha)$ to be the asymptotic class of $\alpha$ in $A$.  By identifying $G$ with the subgroup of inner automorphisms, we have a definition of $\partial_G$ which we abbreviate to $\partial$. Finally, we let $\partial G = \partial_G(G_>)$ be the \emph{set of directions} for $G$.
\end{definition}

We use asymptotic classes to associate a metric space to $G$. This space is the completion of a quotient of $\partial G$ by a pseudometric which we now define.

\begin{definition}
	Suppose $\alpha$ and $\beta$ are automorphisms moving towards infinity and choose $U,V\in\cos(G)$. We define 
	\[\delta_{+,n}^{U,V}(\alpha,\beta) = \min\left\{\left.\dfrac{\log[\alpha^n(U):\alpha^n(U)\cap \beta^k(V)]}{n\log s(\alpha)} \hspace{2pt}\right| k\in\bN_0, s(\beta^k)\le s(\alpha^n)\right\}.\]
	Set 
	\[\delta_{+}(\alpha,\beta) = \limsup_{n\to\infty}\delta_{+,n}^{U,V}(\alpha,\beta).\]
\end{definition}
\begin{proposition}[{\cite[Corollary 16 and Lemma 17]{Baumgartner06}}]\label{intro:prop:pmetric}
	The map $\delta_{+}$ is independent of choice of $U$ and $V$, and takes values between $0$ and $1$. We have $\delta_{+}(\alpha,\beta) = 0$ whenever $\alpha\asymp \beta$. The map $\delta(\alpha,\beta):= \delta_+(\alpha,\beta)+\delta_{+}(\beta,\alpha)$ defines a pseudometric on $\Aut(G)$.
\end{proposition}
Lemma \ref{lem:dir_inv_distance} can be shown by choosing $U$ tidy for $\alpha$ and expanding the definition of $\delta_{+,n}^{U,U}(\alpha,\alpha^{-1})$.
\begin{lemma}[{\cite[Lemma 15]{Baumgartner06}}]\label{lem:dir_inv_distance}
	Suppose $\alpha\in \Aut(G)$ such that both $\alpha$ and $\alpha^{-1}$ move towards infinity. Then $\delta(\alpha,\alpha^{-1}) = 2$.
\end{lemma}

\begin{definition}
	The \emph{space of directions} of $G$ is the completion of the metric space $\partial G/\delta^{-1}(0)$.
\end{definition}
The pseudometric $\delta$ is designed to be a notion of angle between two asymptotic classes. This is analogous to the Tit's metric on the boundary of a $\CAT(0)$ space. Intuition suggests that an $n$-flat, which we interpret as something quasi-isometric to $\bR^n$, inside our group should give an $(n-1)$-sphere inside our boundary. This is the case for appropriate notion of $n$-flat, see the calculations in \cite[Sections 3.4 and 4.2]{Baumgartner06}. On the other hand, we expect hyperbolic groups, which are far from having any $2$-dimensional flats, to have very large angles between rays. This is known to be the case when $G$ is the automorphism group of a regular tree, see \cite[Section 5.1.5]{Baumgartner04}.
We verify this intuition for arbitrary hyperbolic groups in Theorem \ref{thm:mainthm}. 

We give an example of a group with space of directions consisting of two elements. In contrast, Corollary \ref{cor:inf_dir} shows that this cannot be the case for hyperbolic groups.

\begin{example}\label{ex:2_dir}

	Fix a finite abelian group $F$ and let $G_0, G_1$ be two copies of the subgroup of $F^{\bZ}$ given by
	\[\{f = (f_i)_{i\in\bZ}\in F^{\bZ}\mid f_i = \id_F \hbox{ for }i \hbox{ large}\}.\]
	
	Then $G_0$ and $G_1$ are non-compact abelian totally disconnected locally compact groups with basis of compact open subgroups given by $U_N = \{f\in G_k\mid f_i = \id_F \hbox{ for }i \ge N\}$, $N\in\bZ$. 
	
	Set $G = G_0\times G_1$. We use $+$ and $-$ to denote the group operations in $G$. Define the automorphism $\alpha\in \Aut(G)$ as follows: if $(f_i,g_i)_{i\in\bZ}\in G$, then $\alpha((f_i,g_i)_{i\in\bZ})_i = (f_{i + 1}, g_{i - 1})$. We consider the semidirect product $G\rtimes \langle \alpha \rangle$. Group elements are pairs $(g, \alpha^n)$, here $n\in\bZ$ and $g\in G$, and multiplication is given by $(g,\alpha^n)(g', \alpha^m) = (g + \alpha^n(g'), \alpha^{n+m})$. The topology is the inherited from the product topology. We calculate the space of directions. Note that
	\begin{align*}
	(g, \alpha^n)(U_0, \id)(g, \alpha^n)^{-1} & = (g, \alpha^n)(U_0, \id)(-\alpha^{-n}(g), \alpha^{-n})\\ & = (g + \alpha^n(U_0) - g, \id)\\
	& = (\alpha^n(U_0), \id).
	\end{align*}
	Indeed, $(U_0, \id)$ is tidy for $(g, \alpha^n)$ and $s(g, \alpha^n) = |F|^{|n|}$. We see that there are precisely two asymptotic classes, namely $\partial (\id, \alpha)$ and $\partial(\id, \alpha^{-1})$, by considering rays based at $(U_0, \id)$. These classes are distance two apart by Lemma \ref{lem:dir_inv_distance}.
\end{example}
\subsection{Hyperbolic metric spaces}
\label{ssec:HyperbolicMetricSpaces}

We give a short introduction on hyperbolic spaces focusing on results which we rely on later. It is enough for our purposes to restrict ourselves to geodesics. Many of the following results can be generalised to other notions of paths. For example, in \cite{Vaisala05} many results are proved for $h$-short geodesics, $0$-short geodesics correspond to geodesics, and in \cite{Bridson99} many results are proved for quasi-geodesics.

A \emph{geodesic} in a metric space $X$ is an isometric embedding $\gamma:U\to X$,  here $U\in \{[0,t], [0,\infty), \bR\}$. We say $\gamma$ is \emph{infinite} if $U = [0,\infty)$ and \emph{bi-infinite} if $U = \bR$. We abuse notation by identifying $\gamma$ with its image. A \emph{geodesic triangle} is a triple of geodesics with endpoints corresponding to every two element subset of a triple of points in $X$. A metric space is \emph{geodesic} if there exists a geodesic between any two points. In \cite{Gromov87}, Gromov introduced the definition of $\delta$-hyperbolic metric spaces to generalise the metric properties of hyperbolic space and trees.

\begin{definition}
	\label{def:SlimTriangHyp}
	Given $\delta\ge 0$, a geodesic metric space $X$ is \emph{$\delta$-hyperbolic} if for every triple of geodesics $\gamma_0,\gamma_1,\gamma_2$, which form a geodesic triangle, $\gamma_0$ is contained within the $\delta$-neighbourhood of $\gamma_1\cup \gamma_2$.
\end{definition}

Examples of $\delta$-hyperbolic spaces include trees, bounded geodesic metric spaces, standard hyperbolic space and the Coxeter complex of a hyperbolic Coxeter group.

\begin{definition}\label{def:gromovHyp}
	Given $x,y,p\in X$ where $X$ is a metric space, the \emph{Gromov product} of $x$ and $y$ based at $p$ is given by 
	\[(x\mid y)_p = \dfrac{1}{2}\left(d(x,p)+d(y,p) - d(x,y)\right).\]
\end{definition}

There are dual approaches to understanding $\delta$-hyperbolic metric spaces. One of which is by uses of geodesics as seen in Definition \ref{def:SlimTriangHyp}. The second approach is sequential and is highlighted by Theorem \ref{thm:InterpOfGromovProd}. Each approach has its advantages and we use both when required. 

\begin{theorem}[{\cite[Chapter III.H Proposition 1.22]{Bridson99}}]\label{thm:InterpOfGromovProd}
	Suppose $X$ is a geodesic metric space. There exists $\delta \ge 0$ such that $X$ is $\delta$-hyperbolic if and only if there exists $\delta'\ge 0$ such that
	\begin{equation}\label{eq:GromovProdCond}
	(x\mid y)_p\ge \min\{(x\mid z)_p,(y\mid z)_p\} - \delta'
	\end{equation}
	for all $x,y,z,p\in X$.
\end{theorem}
\begin{remark}
	Using the characterisation given by Theorem \ref{thm:InterpOfGromovProd}, the definition of $\delta$-hyperbolic can be extended to arbitrary metric spaces which may not be geodesic.
\end{remark}
Henceforth, if $X$ is a $\delta$-hyperbolic metric space, then we assume that $\delta$ is chosen large enough such that both Definition \ref{def:SlimTriangHyp} and Equation \eqref{eq:GromovProdCond} hold.

We have the following interpretation of the Gromov product. 
\begin{theorem}[{\cite[2.33 Standard Estimate]{Vaisala05}}]\label{thm:stdest}
	Suppose $X$ is a $\delta$-hyperbolic metric space, $p,x,y\in X$ and $\gamma$ a geodesic from $x$ to $y$. Then 
	\[d(p,\gamma) - 2\delta \le (x\mid y)_p \le d(p,\gamma).\]
\end{theorem}

Much of the study of $\delta$-hyperbolic spaces is done via the hyperbolic boundary. This is a topological space associated to every $\delta$-hyperbolic space. Like the definition of $\delta$-hyperbolic, there is more than one way to define the boundary. We use two approaches, the first is by infinite geodesics and the second is by sequences. See \cite[Chapter III.H]{Bridson99} for justification.

\begin{definition}Suppose $X$ is a $\delta$-hyperbolic metric space. Two infinite geodesics $\gamma_1$ and $\gamma_2$ in $X$ are said to be \emph{asymptotic} if $d(\gamma_1(t),\gamma_2(t))$ is bounded as $t\in [0,\infty)$ varies. Equivalently, the Hausdorff distance between the two geodesics is finite, see \cite[Chapter III.H Lemma 3.3]{Bridson99}. This is an equivalence relation on the set of all infinite geodesics. We call the collection of equivalence classes the \emph{hyperbolic boundary} and denote it by $\partial X$.
\end{definition}
\begin{definition}\label{def:seq_grom_boundary}
	Suppose $u = (u_i)_{i\in\bN}$ is a sequence of points in a $\delta$-hyperbolic metric space $X$. We say $u$ \emph{converges at infinity} if $(u_i\mid u_j)_p\to \infty$ as $i,j\to\infty$ for some (hence any) $p\in X$. Two sequences $(u_i)_{i\in\bN}$ and $(v_i)_{i\in\bN}$ are said to be \emph{equivalent} if $(u_i\mid v_j)\to\infty$ as $i,j\to\infty$ for some (hence any) $p\in X$. This is an equivalence relation. Denote the equivalence class of $(u_i)_{i\in\bN}$ by $\lim u_i$ and the set of equivalence classes by $\partial_s X$.
\end{definition}

Any infinite geodesic $\gamma$ in $X$ corresponds to a sequence converging at infinity, namely $(\gamma(n))_{n\in\bN}$. It can be seen that if $\gamma_1$ is asymptotic to $\gamma_2$, then $\lim \gamma_1(n) = \lim \gamma_2(n)$. Hence we have an injective map $\partial X \to \partial_s X$. This map is also a surjection.
\begin{theorem}[{\cite[Chapter III.H Lemma 3.13]{Bridson99}}]
	For $X$ a $\delta$-hyperbolic metric, there exists a bijection $\partial X \to \partial_s X$.
\end{theorem}
We can view $X\cup \partial X$ as a compactification of $X$. We outline the topology on $\partial X$ here following \cite[Capter III.H Section 3]{Bridson99} where justifications can be found. A \emph{generalised geodesic} $\smash{\overline{\gamma}}:[0,\infty]\to X\cup \partial X$ is defined from a geodesic $\gamma:I\to X$, where $I\in\{[0,t],[0,\infty)\mid t\in \bR\}$, in the following way:
\begin{enumerate}
	\item If $I = [0,t]$, then $\overline{\gamma}(x) = \gamma(x)$ for $x\le t$ and $\overline{\gamma}(x) = \gamma(t)$ for $x \ge t$;
	\item If $I = [0,\infty)$, then define $\overline{\gamma}(x) = \gamma(x)$ for $x\in I$ and $\gamma(\infty)$ to be the equivalence class of $\gamma$ in $\partial X$.
\end{enumerate}

Fix $p\in X$. We declare that a sequence $(x_n)_{n\in\bN}$ in $X\cup\partial X$ converges to $x\in X\cup\partial X$ if and only if there exists a sequence of generalised geodesics $(\gamma_n)_{n\in\bN}$ such that
\begin{enumerate}[label = (\roman*)]
	\item $\gamma_n(0) = p$ and $\gamma_n(\infty) = x_n$;
	\item every subsequence of $(\gamma_n)_{n\in\bN}$ contains a subsequence which converges (uniformly on compact sets) to some generalised geodesic $\gamma$ with $\gamma(\infty) = x$.
\end{enumerate}
This defines a topology on $X\cup\partial X$ independent of $p$. Given an infinite geodesic $\gamma$, Real number $k >2\delta$ and $n\in\bN$, define $V_n(\gamma)$ to be the collection of boundary points that admit a representative $\gamma'$, recall that boundary points can be viewed as equivalence classes of infinite geodesics, such that $\gamma(0) = \gamma'(0)$ and $d(\gamma(n),\gamma'(n))\le k$. Then $\{V_n(\gamma)\mid n\in\bN\}$ is a neighbourhood basis, not necessarily open, at $\gamma(\infty)$ in $\partial X$. 

\begin{theorem}[{\cite[Section III.H Theorem 3.9]{Bridson99}}]
	Suppose $f:X\to Y$ is a quasi-isometry between $\delta$ hyperbolic spaces. Then $f$ induces a homeomorphism from $\partial X$ to $\partial Y$.
\end{theorem}
The following Lemma is immediate from Definition \ref{def:seq_grom_boundary}.
\begin{lemma}\label{lem:GromovBound}
	If $\lim v_n = \lim u_n$ for sequences $(u_n),(v_n)\subset X$, then $(u_n\mid u_m)_{p}$ is unbounded for all $p\in X$. 
\end{lemma}

Theorem \ref{thm:RibbonLemma} is a generalisation of the well know result that geodesics between two points in $\delta$-hyperbolic space stay close. This proof is two applications of \cite[Section III.H Lemma 1.15]{Bridson99}.

\begin{theorem}\label{thm:RibbonLemma}
	Suppose $\gamma_0:[0,t_0]\to X$ and $\gamma_1:[0,t_1]\to X$ are two geodesics in $X$, a $\delta$-hyperbolic metric space. Then every point on $\gamma_0$ is within distance
	\[8\delta+ 2d(\gamma_0(0),\gamma_1(0))+ 2d(\gamma_0(t_0),\gamma_1(t_1))\]
	of a point on $\gamma_1$.
\end{theorem}
\begin{proof}
	Choose a geodesic $\gamma:[0, t_3]\to X$ from $\gamma_0(0)$ to $\gamma_1(t_1)$. Then any point on $\gamma_0$ is within distance $2(d(\gamma_0(t_0),\gamma_1(t_1))+2\delta)$ of $\gamma$ by \cite[Section III.H Lemma 1.15]{Bridson99}. Applying the same lemma to the geodesics travelling along  $\gamma$ and $\gamma_1$ but in the opposite direction, every point on $\gamma$ is within distance $2(d(\gamma_0(0),\gamma_1(0))+2\delta)$ of a point in $\gamma_1$. Thus every point on $\gamma_0$ is within distance 
	\[8\delta+ 2d(\gamma_0(0),\gamma_1(0))+ 2d(\gamma_0(t_0),\gamma_1(t_1))\]
	of a point on $\gamma_1$.
\end{proof}
Given two geodesics between two points, Theorem \ref{thm:RibbonLemma} can be improved. Again, this result follows from \cite[Section III.H Lemma 1.15]{Bridson99}.
\begin{theorem}\label{thm:fellow_travel}
	Suppose $\gamma_0,\gamma_1:[0,t]\to X$ are two geodesics in a $\delta$-hyperbolic space with $\gamma_0(0) = \gamma_1(0)$ and $\gamma_1(t) = \gamma_2(t)$. Then $d(\gamma_0(t'),\gamma_1(t'))\le 4\delta$ for all $t'\in [0,t]$.
\end{theorem}
The following theorem follows from \cite[Section III.H Lemma 3.2 and Lemma 3.3]{Bridson99} and Theorem \ref{thm:RibbonLemma}.
\begin{theorem}\label{thm:BoundaryIsVisual}
	Given two points $\varepsilon_{1}$ and $\varepsilon_2$ on the boundary of a $\delta$-hyperbolic metric space, there exists a bi-infinite geodesic $\gamma$ from $\varepsilon_1$ to $\varepsilon_2$, that is, $\lim_{n\to\infty}\gamma(n) = \varepsilon_1$ and $\lim_{n\to\infty}\gamma(-n) = \varepsilon_2$. Any two such geodesics are Hausdorff distance at most $28\delta$ apart. 
\end{theorem}
\subsection{Hyperbolic topological groups}
\label{ssec:HyperbolicTopologicalGroups}
We study topological groups with a vertex transitive action on a locally finite connected graph $\Gamma$ such that vertex stabilisers are compact open. Theorem \ref{thm:RCay} shows that these groups are not uncommon and are precisely compactly generated t.d.l.c. groups. We give an outline of the proof for completeness. First, we establish our language for graphs.

For us, a \emph{graph} $\Gamma$ is a disjoint union of countable sets $V(\Gamma)\sqcup E(\Gamma)$ where $E(\Gamma)$ is a symmetric diagonal-free subset of $V(\Gamma)\times V(\Gamma)$. We call elements of $V(\Gamma)$ \emph{vertices} and elements of $E(\Gamma)$ \emph{edges}. Given an edge $(v_0, v_1)$, we set $\overline{(v_0, v_1)}$ to be the edge $(v_1,v_0)$. We say $v_0$ is \emph{adjacent} to $v_1$ if $(v_0, v_1)$ is an edge. A \emph{path of length $n$} in $\Gamma$ is a sequence of vertices $(v_0, v_1,\ldots, v_n)$ such that $(v_i, v_{i+1})$ is an edge and $v_i\neq v_{i+2}$. The distance between two vertices is the length shortest path between them. We extend the definition of geodesic to graphs by identifying edges with the unit interval. Thus, a geodesic in a graph is a path between two vertices of minimal length. A graph is \emph{connected} if it is a geodesic  metric space and \emph{locally finite} if each vertex is adjacent to only finitely many other vertices. We will make use of colourings on edges of graphs. We assume that any colouring $c:E(\Gamma)\to \bN$ satisfies $c(e) = c(\overline{e})$.

\begin{theorem}[{\cite{Abels74}, \cite[Corollary 1]{Moller03}, \cite[Theorem 2.2]{Kron08}}]\label{thm:RCay}
	Suppose $G$ is a compactly generated topological group with a compact open subgroup. There exists a locally finite connected graph $\Gamma$ such that $G$ acts vertex transitively on $\Gamma$ with compact open vertex stabilisers.	
	
	Conversely, any group $G$ acting vertex transitively on a locally finite connected graph with compact open vertex stabilisers is compactly generated.
\end{theorem}
\begin{proof}[Proof outline]
	Suppose $G$ is a compactly generated topological group with a compact open subgroup $U$. Let $S$ be a symmetric compact generating set. Define a graph $\Gamma$ as follows:
	\begin{align*}
	V(\Gamma) &= G/U\\
	E(\Gamma) &= \{(gU,gsU)\mid g\in G, s\in S\}.
	\end{align*}
	Clearly $G$ acts transitively on $V(\Gamma)$ by left multiplication. That $\Gamma$ is connected follows as $S$ is a generating set. Local finiteness of $\Gamma$ follows from a compactness argument. The stabiliser of the vertex $gU$ is $gUg^{-1}$ which is compact open.
	
	Conversely, suppose $G$ acts vertex transitively on a locally finite connected graph $\Gamma$ with compact open vertex stabilisers. Choose any vertex $v\in \Gamma$. For each vertex $u$ adjacent to $v$, choose $g_u\in G$ with $g_u(v)= u$. Then $G_v\cup \{g_u,g_u^{-1}\mid u \hbox{ is adjacent to }v\}$ is a compact generating set for $G$.
\end{proof}

\begin{definition}
	Suppose $G$ is a compactly generated topological group containing a compact open subgroup. Call any connected locally finite graph on which $G$ acts vertex transitively with compact open vertex stabilisers a \emph{rough Cayley graph} for $G$.
\end{definition}
\begin{remark}
	Rough Cayley graphs appear throughout literature but are sometimes called Cayley-Abels graphs or relative Cayley graphs. They are also a special case of a Schreier graph and sometimes carry that label. 
\end{remark}
\begin{theorem}{\cite[Theorem 2.7]{Kron08}}
	Suppose $G$ is a compactly generated topological group containing a compact open subgroup. Then any two rough Cayley graphs are quasi-isometric.
\end{theorem}
Rough Cayley graphs allow the study of compactly generated locally compact groups using geometric group theory. See \cite{Kron08} for analogues of Stalling's Theorem and accessibility. It is natural to consider the relationship between the geometric group theory of these groups and the general topological group theory. For example, in \cite{Baumgartner12} it is shown that every t.d.l.c. group with hyperbolic rough Cayley graph has flat rank at most $1$.

Observe that since being $\delta$-hyperbolic (up to change in $\delta$) is a quasi-isometry invariant: if $G$ has a $\delta$-hyperbolic rough Cayley graph, then any rough Cayley graph of $G$ is $\delta'$-hyperbolic for some $\delta'\ge 0$.

\begin{definition}
	A \emph{hyperbolic group} is a topological group with $\delta$-hyperbolic rough Cayley graph. 
\end{definition}

The following theorem provides a source of examples which include groups with sufficiently transitive action on hyperbolic buildings.
\begin{proposition}[{\cite[Proposition 14]{Baumgartner12}}]
	Suppose $G$ is a t.d.l.c. group acting cocompactly with compact open point stabilisers on a connected  $\delta$-hyperbolic $M_k$-polyhedral complex. Then $G$ is a hyperbolic group.
\end{proposition}
There are only a few possibilities for isometries of $\delta$-hyperbolic metric spaces.
\begin{definition}\label{def:isom_class}
	Suppose $g$ is an isometry of a $\delta$-hyperbolic metric space $X$. 
	\begin{enumerate}[label = (\roman{*})]
		\item If there is point $x\in X$ such that the orbit of $x$ under the action of $g$ is bounded, we say $g$ is \emph{elliptic}.
		\item If $g$ is not elliptic and there exists a unique $\varepsilon\in\partial X$ fixed by $g$, then we say $g$ is \emph{parabolic}.
		\item If $g$ is not elliptic but $g$ fixes precisely two distinct points in $\partial X$, we call $g$ \emph{hyperbolic}. 
	\end{enumerate}
	
	For a compactly generated t.d.l.c. group, we say $g\in G$ is \emph{elliptic}, \emph{parabolic} or \emph{hyperbolic} if $g$ is an elliptic, parabolic or hyperbolic automorphism of some rough Cayley graph. It can be shown that this is independent of choice of rough Cayley graph using the quasi-isometry between any two rough Cayley graphs.
\end{definition}
\begin{theorem}{\cite[Theorem 1 and Corollary 4]{Woess93}}\label{thm:class_hyp_action}
	Suppose $g$ is an isometry of a $\delta$-hyperbolic metric space $X$. Then $g$ is either elliptic, parabolic or hyperbolic. If $g$ is hyperbolic, then the two unique ends fixed by $g$ can be recovered as
	\[\omega_{+}(g) = \lim_{n\to\infty}g^n(x)\hbox{ and } \omega_{-}(g) = \lim_{n\to\infty}g^{-n}(x),\]
	where $x\in X$.
\end{theorem}

Typically, parabolic automorphisms only exist in spaces that are in some sense large. An obstruction to their existence is the if isometry group of the space does acts properly and cocompactly. 

\begin{theorem}[{\cite[Theorem 22]{Baumgartner12}}]
	A hyperbolic group does not contain any parabolic elements.
\end{theorem}

We study asymptotic classes of elements in a t.d.l.c. hyperbolic groups. It is useful to know that elliptic elements do not move towards infinity.

\begin{proposition}[{\cite[Proposition  20]{Baumgartner12}}]\label{prop:EllipAndUniscalar}
	Suppose $G$ is a t.d.l.c. hyperbolic group and $g\in G$ is elliptic. Then $g$ is uniscalar.
\end{proposition}
\begin{corollary}\label{cor:ell_and_inf}
	Suppose $G$ is a t.d.l.c. hyperbolic group and $g\in G$ is elliptic. Then $g$ does not move towards infinity.
\end{corollary}
\begin{proof}
	Uniscalar elements do not move towards infinity as they stabilise a compact open subgroup, see Proposition \ref{prop:scale_prop}. Elliptic elements are uniscalar by Proposition \ref{prop:EllipAndUniscalar}.
\end{proof}

\section{An invariant geodesic for a hyperbolic group element}\label{sec:InvariantGeodesic}

The purpose of this section is to prove the following theorem which is of independent interest.
\begin{theorem}\label{thm:axis}
	Suppose $g$ is a hyperbolic automorphism of a locally finite $\delta$-hyperbolic graph $\Gamma$. Then there exists a bi-infinite geodesic $\gamma$ in $\Gamma$ and $n\in\bN$ such that $g^n$ acts by translation along $\gamma$. 
\end{theorem} 

Our argument is based on a proof found in \cite{Delzant91} which proves the same result but for a finitely generated hyperbolic group acting on a usual Cayley graph. The main complication in generalising to the non-discrete case is the lack of $G$-invariant colouring of the rough Cayley graph. 

Suppose $g$ is a hyperbolic isometry of a locally finite $\delta$-hyperbolic graph $\Gamma$. Let $A'(g)$ to be the collection of all geodesics from $\omega_{-}(g)$ to $\omega_{+}(g)$, of which there is at least one by Theorem \ref{thm:BoundaryIsVisual}. Let $A(g)$ be the union of all geodesics in $A'(g)$. We say a geodesic $\gamma$ in $\Gamma$ is \emph{$g$-minimal} if it is a segment (sub-geodesic) of a geodesic in $A'(g)$.

\begin{lemma}\label{lem:ballsonaxis}
	Suppose $g$ is a hyperbolic automorphism of a locally finite $\delta$-hyperbolic graph $\Gamma$ and $v\in V(A(g))$. There exists a ball $B$ centred at $v$ and $n\in\bN$ such that:
	\begin{enumerate}[label = (\roman{*})]
		\item $B$ meets every geodesic in $A'(g)$\label{itm:Ai};
		\item $g^{ni}(B)\cap g^{nj}(B)\cap A(g)\neq \varnothing$ if and only if $i = j$\label{itm:Aii};
		\item Any $g$-minimal geodesic from a vertex  in $g^{ni}(B)\cap A(g)$ to a vertex in $g^{nj}(B)\cap A(g)$, where $i\le j$, meets $g^{nk}(B)$ for all $i \le k \le j$\label{itm:Aiii};
		\item $E(A(g))/\langle g^n\rangle$ is finite\label{itm:Aiv}.
	\end{enumerate} 
\end{lemma}

\begin{proof}
	Choose a geodesic $(\ldots, v_{-1}, v_0, v_1,\ldots)\in A'(g)$ such that 
	\[v = v_0\hbox{,\ \ }\lim g^n(v) = \lim v_n\hbox{\ \ and\ \ }\lim g^{-n}(v) = \lim v_{-n}.\] Lemma \ref{lem:GromovBound} shows that $(g^{-n}(v)\mid g^{m}(v))_{v}$ is bounded by some $K\in\bN$. By enlarging $K$ if necessary, we assume that $K> 8\delta+2(28\delta+28\delta)$. Let $B$ be the ball of radius $K$ about $v$. Theorem \ref{thm:BoundaryIsVisual} gives \ref{itm:Ai}.
	
	Lemma \ref{lem:GromovBound} gives $m,N\in\bN$ such that $n\ge N$ implies $(g^{n}(v)\mid v_{m})_v > 3K$ and $(g^{-n}(v)\mid v_{-m})_v > 3K$. For each $n\in\bZ$ choose $f(n)\in\bZ$ such that $d(g^{n}(v),v_{f(n)})$ is minimised. Since $gA(g) = A(g)$, it follows that $d(g^n(v), v_{f(n)})\le 28\delta< K$ by Theorem \ref{thm:BoundaryIsVisual}. 
	
	\begin{claim} 
		The function $f:\bZ\to\bZ$ satisfies the following properties:
		\begin{enumerate}[label = (\alph{*})]
			\item\label{itm:1_f(n)} if $n > N$, we have $f(n)\ge 2K$;
			\item\label{itm:2_f(n)} if $n < -N$, we have $f(n)\le -2K$.
		\end{enumerate}
	\end{claim}
	\begin{proof}
		We prove the contrapositive of \ref{itm:1_f(n)}. Property \ref{itm:2_f(n)} follows via a symmetric argument. First, suppose $0\le f(n)\le 2K$. Then $v$ is within distance $2K$ of any geodesic from $v_{f(n)}$ to $g^n(v)$, indeed, $d(v,v_{f(n)}) = f(n)\le 2K$. Theorem \ref{thm:stdest} shows ${(v_{f(n)}\mid v_n)_v\le 2K}$. By choice of $N$, we must have $n \le N$. Alternatively, now suppose $f(n)< 0$. Since $d(g^n(v),v_{f(n)}) \le 28\delta$, Theorem \ref{thm:RibbonLemma} shows that any geodesic from $g^n(v)$ to $v_m$ stays within distance ${50\delta< 3K}$ of $(v_{f(n)},\ldots,v_0,\ldots, v_m)$. It follows from Theorem \ref{thm:stdest}  that $(v_{f(n)}\mid v_m)_v \le 3K$. Assumptions on $N$ show that $n\le N$. Thus \ref{itm:1_f(n)} holds.
	\end{proof}
	Since $g$ is hyperbolic, no orbit of $g$ is bounded. Enlarging $N$ if necessary, we assume that $g^N(B)\cap B = \varnothing$. Note that if $u\in g^N(B)\cap A(g)$ and $u'\in g^{-N}(B)\cap A(g)$, then any $g$-minimal geodesic from $u$ to $u'$ stays within distance $2(28\delta+28\delta)+8\delta<K$ of the geodesic $(v_{f(-N)},\ldots v_0,\ldots, v_{f(N)})$ by Theorem \ref{thm:RibbonLemma} and so must cross $B$.
	
	We make the abbreviation $B_{i}:= g^{Ni}(B)$ for $i\in\bZ$. We show that for $i\le j$, any $g$-minimal geodesic between vertices in $B_{i}$ and $B_{j}$, must cross $B_{k}$ for $i \le k\le j$. Replacing $B_i$ and $B_j$ with $g^{-Ni - 1}(B_i) = B_{-1}$ and $g^{-Ni-1}(B_j)$, we assume $i = -1$. We proceed by induction on $j$. We have already verified the case when $- 1\le j\le 1$. Suppose that for fixed $j\in\bN$, if $-1 \le k \le j$, any $g$-minimal geodesic from a vertex in $B_{-1}$ to $B_{j}$ crosses $B_{k}$. Suppose $u\in B_{-1}$, $u'\in B_{j+1}$, and let $\gamma$ be a $g$-minimal geodesic between them. Then $\gamma$ stays within distance $2(28\delta+28\delta)+8\delta$ of the geodesic $(v_{f(-N)},\ldots,v_0,\ldots v_{f(N(j+1)})$ by Theorem \ref{thm:RibbonLemma}. This implies $\gamma$ crosses $B$. Let $\gamma'$ be a segment of $\gamma$ from a vertex in $B$ to $u'$. Then $g^{-N}(\gamma')$ is a $g$-minimal geodesic from a vertex in $B_{-1}$ to a vertex in $B_{j}$. The inductive hypothesis shows that $g^{-N}(\gamma')$ crosses $B_{k}$ for $-1\le k \le j$. Hence, $\gamma'$ crosses $B_{k}$ for $0\le k\le j+1$. This shows $\gamma$ crosses $B_k$ for $-1\le k\le j+1$. This completes the induction. We have shown \ref{itm:Aiii}.
	
	For \ref{itm:Aiv}, suppose that $(u_1,u_2)\in E(A(g))$. Let $\gamma = (\ldots,u_{-1},u_0, u_1,\ldots)\in A'(g)$ be a geodesic containing $(u_1,u_2)$. Note that $B_i\cap \gamma$ is non-empty and disjoint for $i\neq j$. Hence, the set $\{j\in\bN\mid u_i\in \cup_{j\in\bZ} B_j\}$ is not bounded above or below. Choose $j_0\le 0$ maximal and $j_1\ge 1$ minimal such that $u_{j_k}\in B_{i_k}$ for some $i_0,i_1\in\bZ$. Thus, $(u_0,u_1)$ is on a geodesic between vertices in $B_{i_0}$ and $B_{i_1}$. Applying \ref{itm:Aiii}, we see that $i_0\le i_1\le i_0+1$. Thus, $g^{-i_0N}(u_0,u_1)$ is on a geodesic between vertices in $B_{0}$ and $B_0$ or $B_0$ and $B_1$. There are only finitely many such geodesics since $\Gamma$ is locally finite and $B_0$ and $B_1$ are both finite. 
\end{proof}

\begin{lemma}\label{lem:NoUpperBound}
	Suppose $g$ is a hyperbolic isometry of a $\delta$-hyperbolic metric space $X$ and $F\subset X$ a finite set of points. Then for any $M\in\bN$ there exists $N\in\bN$ such that $n > N$ implies $\min_{x\in F}d(g^n(x),x) > M$. 
\end{lemma}
\begin{proof}
	We prove the converse. Suppose there exists a sequence $(x_i)_{i\in\bN}\subset F$ and an unbounded sequence $(n_i)_{i\in\bN}\subset \bN$ such that $(g^{n_i}(x_i),x_i)\le M$. Since $V$ is finite, we may assume that $x_i = x\in F$ is constant. Then $d(g^{n_i}(x),x)\le M$. We cannot have $\lim_{n\to\infty}g^n(x)\in\partial X$, hence $g$ is not hyperbolic by Theorem \ref{thm:class_hyp_action}. 
\end{proof}
We now proceed with the proof of Theorem \ref{thm:axis}. We use an inverse limit of finite sets. For the readers benefit, we have included the definition and required properties of inverse limits as Definition \ref{def:inv_sys} and Lemma \ref{lem:inverse_limit_sets} at the end of this section.
\begin{proof}[Proof of Theorem \ref{thm:axis}]
	Replace $g$ with $g^n$ where $n$ is given by Lemma \ref{lem:ballsonaxis}. Choose a set $F$ of edges in $A(g)$ such that $E(A(g))/\langle g\rangle$ admits precisely one representative in $F$. This is a finite set by Lemma \ref{lem:ballsonaxis}. Replacing $g$ with a power of itself if necessary, Lemma \ref{lem:NoUpperBound} shows that we may suppose that $v\in V(F)$ implies $d(v,g(v))> 2$. Choose a colouring on $E(A(g))/\langle g\rangle$. This induces a colouring on $E(A(g))$ which is preserved by the action of $g$. To see that each vertex meets each colour at most once, suppose for the sake of contradiction that $e_0 = (v_0,v)$ and $e_1 = (v_1, v)$ are distinct edges with the same colour. Then we must have $g^k(e_0) \in \{e_1, \overline{e}_1\}$ for some $k\in\bZ$. By exchanging $e_0$ and $e_1$ if necessary we may assume that $k \ge 0$. We may also assume $v_0\in F$ by choosing an appropriate image under the action of $g$ and $g^{-1}$. We must have $d(g^k(v_0), v_0)\le 2$. This contradicts choice of $g$. 
	
	By placing a total order on colours we obtain a short-lex ordering on paths in $A(g)$; a path is smaller than another path in this ordering if it is shorter in length or its label induced by the colouring comes first in the dictionary. We call a path in $A'(g)$ short-lex if it is minimal of all such paths between its endpoints. Since each vertex sees each colour at most once, a short-lex path between two vertices is unique. Observe that a subpath of a short-lex path is also short-lex, otherwise the original path can be altered to give a smaller path in the short-lex ordering. We use the shorthand $B_i = g^i(B)$ where $B$ is given in Lemma \ref{lem:ballsonaxis}. For $i\in\bN$ let $X_i$ be the collection of all short-lex paths from a vertex in $B_{-i}$ to a vertex in $B_{i}$. Note that $X_i$ is non-empty since any geodesic in $A'(g)$ induces a path from $B_{-i}$ to $B_{i}$. Given a path in $X_i$ and $j\le i$, this path crosses both $B_{-j}$ and $B_{j}$. Since every segment of a short-lex path is short-lex, by choosing the minimal segment from $B_{-j}$ to $B_{j}$ we define a map $X_i\to X_j$. This defines a projective system and we let $Y$ denote the limit of this system. Observe that since the short-lex path between to vertices is unique, we have $|X_i|\le |B_i|^2$. This is finite and independent of $i$ since $\Gamma$ is locally finite and $|B_i| = |B_j|$ for all $i,j\in\bN$. Lemma \ref{lem:inverse_limit_sets} shows that $Y$ is finite and non-empty.
	
	Finally, given a geodesic $\gamma\in Y$, here we are identifying $\gamma$ with an increasing union of its components, we choose segments of $g(\gamma)$ which run from $B_{-i}$ to $B_{i}$. Since $g$ preserves the colouring, each of these segments lies in $X_i$. Thus, $g(\gamma)\in Y$. We have shown that $g$ permutes $Y$ which is a finite subset of geodesics in $A'(g)$. It follows that some power of $g$ stabilises these geodesics as required.
\end{proof}

\begin{definition}\label{def:inv_sys}
	An \emph{inverse system} of sets $(Y_i,f_{j,k})$ over the natural numbers is a collection of non-empty sets $(Y_i)_{i\in\bN}$ with maps $f_{i,j}:Y_j\to Y_i$, where $i\le j$, such that $f_{i,j}f_{j,k} = f_{i,k}$ for all $i\le j \le k$.
	The \emph{inverse limit} of the inverse system $(Y_i,f_{i,j})$ is the set
	\[Y = \{(y_i)_{i\in\bN}\in\prod_{i\in\bN}Y_i\mid f_{i,j}(y_j) = y_i\}\]
\end{definition}
\begin{lemma}\label{lem:inverse_limit_sets}
	Retain the notation of Definition \ref{def:inv_sys}. If $Y_i$ is finite for each $i\in\bN$, then $Y$ is non-empty. If $|Y_i|\le K$ for all $i\in\bN$, then $|Y|\le K$.
\end{lemma}
\begin{proof}
	Choose any sequence $(y_i)_{i\in\bN}$ in $\bigcup_{i\in\bN} Y_i$. Let $l(y_i)$ be the smallest $j\in\bN$ such that $y_i\in Y_j$. Construct $I_n\subset \bN$ and $x_n\in Y_n$ inductively as follows. First, since $Y_1$ is finite, there exists $x_1\in Y_1$ such that $f_{1,l(y_i)}(y_i)= x_1$ for infinitely many $y_i$. Let 
	\[I_1 = \{i\in\bN\mid f_{1,l(y_i)}(y_i)= x_1\}.\]
	Suppose for some $n\in\bN$, we have elements $x_j\in Y_j$ for $j\le n$ and $I_n\subset \bN$ infinite such that $i\in I_n$ implies $f_{j,l(y_i)} = x_j$ for all $j\le n$. Since $I_n$ is infinite and $Y_{n+1}$ is finite, there exists $x_{n+1}\in Y_{n+1}$ and $I_{n+1}\subset I_n$ infinite such that $i\in I_{n+1}$ implies $f_{n+1,l(y_i)} = x_{n+1}$. We then have $f_{j,n+1}(x_{n+1}) = x_j$ for all $j\le n$. This completes our induction. By construction, we have $(x_n)_{n\in\bN}\in Y$.
	
	Suppose $|Y_i|\le K$ for all $i\in\bN$. Let $Y_i'$ be the image of $Y$ under the projection onto the $i$-th coordinate. The map $i\mapsto |Y_i'|$ is non-decreasing, as $f_{i,j}:Y'_j\to Y'_i$ is a surjection for all $i\le j$.  Hence $|Y_i'| $ is constant for $i$ sufficiently large. Since different sequences in $Y$ can be distinguished by their projection onto sufficiently large $Y'_i$, we must have $Y$ in bijection with $Y_i'$ for $i$ sufficiently large. Hence $Y$ is finite. 
\end{proof}
\section{The space of directions of a hyperbolic group}\label{sec:hypSpaceOfDirections}
In this section we prove Theorem \ref{thm:mainthm} which was given as Conjecture 32 in \cite{Baumgartner12}.
\begin{theorem}\label{thm:mainthm}
	Suppose $G$ is a t.d.l.c. group acting vertex transitively on a $\delta$-hyperbolic graph $\Gamma$ with compact open vertex stabilisers. Then:
	\begin{enumerate}[label = (\roman{*})]
		\item \label{itm:part_1_main_thm}The map which sends a hyperbolic element moving towards infinity to its attracting point on the boundary induces an injection from the set of directions of $G$ to $\partial \Gamma$;
		\item \label{itm:part_2_main_thm}Distinct asymptotic classes have distance $2$ apart in the space of directions metric.
	\end{enumerate}
\end{theorem}
Our proof is split in two. Lemma \ref{lem:dir_to_hyp} uses the invariant geodesic constructed in Theorem \ref{thm:axis} to control distances between vertices in orbits of elements with the same attracting end. This control over distance allows us to bound relevant coset indices and thus show that hyperbolic elements moving towards infinity with the same attracting end are in the same asymptotic class.

We then use hyperbolic geometry in Lemma \ref{lem:hyp_dist_lem} to bound orbits of particular compact open subgroups. Finally, we use this bound to estimate coset index calculations in the space of directions pseudometric. Our estimates show that distinct asymptotic classes are distance $2$ apart. 

\begin{lemma}\label{lem:dir_to_hyp}Retain the notation of Theorem \ref{thm:mainthm}.
	Suppose $g,h\in G$ are hyperbolic with $\omega_+(g) = \omega_+(h)$. Then   $g\asymp h$. In particular, if $g$ moves towards infinity, then so does $h$.
\end{lemma}
\begin{proof}
	Theorem \ref{thm:axis} allows us to assume  that, by replacing $g$ and $h$ by powers of themselves if necessary, there exist bi-infinite geodesics on which $g$ and $h$ act by translation. Let $(u_0,u_1,\ldots)$ and $(v_0,v_1,\ldots)$ be infinite segments of these geodesics such that $\lim u_i =\omega_+(g) = \omega_+(h) = \lim v_i$. These geodesics must be uniformly bounded and so  there exists $M\in\bN$ such that $d(v_n,u_n)\le M$ for $n\in\bN$, see \cite[Chapter III.H Lemma 3.3]{Bridson99}. By taking powers of $h$ and $g$ if necessary we assume that $g(u_n) = u_{n+m}$ and $h(v_n) = v_{n+m}$ for some fixed $m\in\bN$. Then the size of the orbit of $v_{km}$ under the action of $G_{u_{km}} = g^kG_{u_0}g^{-k}$ is bounded by some number depending only $M$. The Orbit-Stabiliser Theorem implies that $[g^kG_{u_0}g^{-k}:g^kG_{u_0}g^{-k}\cap G_{v_{km}}]$ is bounded. Since $G_{v_{km}} = h^{k}G_{v_0}h^{-k}$, we see that  $[g^kG_{u_0}g^{-k}:g^kG_{u_0}g^{-k}\cap h^{k}G_{v_0}h^{-k}]$ is bounded. A symmetric argument shows that $[h^{k}G_{v_0}h^{-k}:h^{k}G_{v_0}h^{-k}\cap g^kG_{u_0}g^{-k}]$ is also bounded. Hence $g\asymp h$. Lemma \ref{lem:asymp_towards_inf} completes the proof.
\end{proof}

\begin{lemma}\label{lem:hyp_dist_lem}
	Retain the notation of Theorem \ref{thm:mainthm} and suppose $g,h\in G$ are hyperbolic such that $\omega_+(g)\neq \omega_+(h)$. Then for each vertex $v\in V(\Gamma)$, there exists $K\in\bN$ such that $|G_{g^n(v),h^m(v)}(v)|\le K$ for all $n,m\in\bN$.
\end{lemma}
\begin{proof}
	Choose a vertex $v\in \Gamma$. Since $\omega_+(g)\neq \omega_+(h)$, we must have $(g^n(v)\mid h^m(v))_v$ bounded by $N\in\bN$. Theorem \ref{thm:stdest} shows that any geodesic from $g^n(v)$ to $h^m(v)$ passes within distance $N+2\delta$ for $v$. Let $p_{m,n}$ be the point on a geodesic from $g^n(v)$ to $h^m(v)$ closest to $v$. If $x\in G_{g^n(v),h^m(v)}$, then $d(p_{m,n},x(p_{m,n})))\le 4\delta$ by Theorem \ref{thm:fellow_travel}. Hence 
	\[d(v,x(v))\le d(v,p_{m,n})+d(p_{m,n},x(p_{m,n}))+d(x(p_{m,n}),x(v)) \le 2(N +2\delta)+4\delta =: N'.\]
	Note that $N'$ does not depend on $m$ or $n$. Also the ball of radius $N'$ about $v$ contains only finitely many vertices. Let this value be $K$. Then $|G_{g^n(v),h^m(v)}(v)|\le K$.
\end{proof}

\begin{proof}[Proof of Theorem \ref{thm:mainthm}]
	Suppose $g,h\in G$ are hyperbolic with $\omega_{+}(g)\neq\omega_{+}(h)$ and are both moving towards infinity. We show that the asymptotic classes containing $g$ and $h$ have pseudo-distance $2$, in particular they are distinct. This combined with Lemma \ref{lem:dir_to_hyp} completes the result. Fix a vertex $v\in V(\Gamma)$. Lemma \ref{lem:hyp_dist_lem} gives $K\in\bN$ such that $|G_{g^n(v),h^m(v)}(v)|\le K$ for all $m,n\in\bN$. For each pair $m,n\in\bN$, the collection of cosets $\{xG_{g^n(v),h^m(v)}(v)\mid x\in G_{g^n(v)}\}$ is a covering of $G_{g^n(v)}(v)$ by sets with size $|G_{g^n(v),h^m(v)}(v)|$. We must have  
	\[\left|\{xG_{g^n(v),h^m(v)}(v)\mid x\in G_{g^n(v)}\} \right|\ge\dfrac{|G_{g^n(v)}(v)|}{|G_{g^n(v),h^m(v)}(v)|}.\]
	Furthermore, if $x_i,x_j\in G_{g^n(v)}$ with  $x_iG_{g^n(v),h^m(v)}(v)\neq x_jG_{g^n(v),h^m(v)}(v)$, then we have $(x_i)^{-1}x_j\not\in G_{g^n(v),h^m(v)}$. Hence 
	\[[G_{g^n(v)}: G_{g^n(v),h^m(v)}]\ge \dfrac{|G_{g^n(v)}(v)|}{|G_{g^n(v),h^m(v)}(v)|}\ge \dfrac{|G_{g^n(v)}(v)|}{K}.\]  
	Now 
	\[|G_{g^n(v)}(v)| = [G_{g^n(v)}: G_{g^{n}(v),v}]\hbox{ and }[G_{g^n(v)}: G_{g^n(v),h^m(v)}] = [G_{g^n(v)}: G_{g^n(v)}\cap G_{h^m(v)}].\]
	It follows from the definition of the scale function that
	\[[G_{g^n(v)}:G_{g^n(v),v}] = [G_{g^n(v)}:G_{g^n(v)}\cap G_v] = [g^nG_vg^{-n}: g^nG_vg^{-n}\cap G_v]\ge s(g)^n\]
	and so
	\begin{align*}
	1&\ge \delta_+(g,h)\\ &= \limsup_{n\to\infty}\min\left\{\dfrac{\log[G_{g^n(v)}:G_{g^n(v),h^{m}(v)}]}{n\log s(g)}\mid s(h^m)\le s(g^n)\right\}\\
	& \ge \limsup_{n\to\infty}\dfrac{\log[G_{g^n(v)}:G_{g^n(v),v}] - \log(K)}{n\log s(g)}\\
	& \ge \limsup_{n\to\infty}\dfrac{\log(s(g)^n)}{n\log s(g)} - \limsup_{n\to\infty}\dfrac{\log(K)}{n\log s(g)}\\
	& = 1.
	\end{align*}
	A symmetric argument shows $\delta_{+}(h,g) =1$ and so $\delta(g,h) = 2$.
\end{proof}

We investigate some corollaries which arise from Theorem \ref{thm:mainthm}. First, \cite[Theorem 3]{Baumgartner12} gives \cite[Theorem 1]{Baumgartner12} as a corollary.

\begin{corollary}
	Suppose $G$ is a t.d.l.c. hyperbolic group. The flat rank of $G$ is at most $1$.
\end{corollary}

We derive some results on $|\partial G|$ and the image of the map given in Theorem \ref{thm:mainthm}. We will use orbits in the boundary of a rough Cayley graph. Lemma \ref{lem:scale_to_boundary} gives a condition on when an orbit is large.

\begin{lemma}\label{lem:scale_to_boundary}
	Suppose $G$ is a hyperbolic t.d.l.c. group with rough Cayley graph $\Gamma$ and $g\in G$ such that $s(g)>1$. Then the orbit of $\omega_{-}(g)\in \partial \Gamma$ under the action of $G$ is infinite.
\end{lemma}
\begin{proof}
	Observe that $g$ is hyperbolic by Corollary \ref{cor:ell_and_inf} and so $\omega_-(g)$ is well defined. Fix $K\in\bN$. We show that $|G\omega_{-}(g)|\ge K$. To do so let $M$ be the number of vertices within distance $2\delta$ of some vertex $v\in \Gamma$. Since $G$ acts vertex-transitively on $\Gamma$, this number is independent of $v$.
	
	Since $s(g)^n = s(g^n)$ and $\omega_{-}(g) = \omega_{-}(g^n)$, Theorem \ref{thm:axis} allows us to assume that $g$ acts by translation along some bi-infinite geodesic. Denote this geodesic by $(v_{i})_{i\in\bZ}$ with $\lim v_{-i}\to \omega_{-}(g)$. 
	
	Observe that for all $k\in\bN$ that 
	\[s(g)^k \le [G_{v_0}:G_{v_0}\cap g^{-k}G_{v_0}g^k] = [G_{v_0}:G_{v_0}\cap G_{g^{-k}(v_0)}].\]
	The orbit-stabiliser theorem shows that the size of the orbit of $g^{-k}(v_0)$ under $G_{v_0}$ is unbounded as $k\to\infty$. We assume that $k$ is large enough so that the size of this orbit is larger that $MK$. Choice of $M$ gives $\{h_i\mid 0\le i\le K\}\subset G_{v_0}$ such that $d(h_ig^{-k}(v_0), h_jg^{-k}(v_0)) > 2\delta$ for $i\neq j$.  Noting that $(h(v_{-i}))_{i\in\bN}$ is a geodesic from $v_0$ to $h^{-i}(\omega_{-}(g))$ which contains $g^{-k}(v_0)$, \cite[Chapter III.H Lemma 3.3]{Bridson99} shows that $h_i(\omega_-(g))\neq h_j(\omega_-(g))$ for $i\neq j$. This completes the proof. 
\end{proof}
We characterise completely when the space of directions of a hyperbolic group consists of a single point.
\begin{corollary}
	Suppose $G$ is a hyperbolic t.d.l.c. group with rough Cayley graph $\Gamma$. Then $|\partial G| = 1$ if and only if:
	\begin{enumerate}
		\item $G$ is not uniscalar;
		\item $G$ fixes a point in $\partial \Gamma$.
	\end{enumerate}
\end{corollary}
\begin{proof}
	Suppose $|\partial G| = 1$. There exists $g\in G$ such that $s(g) > 1$, thus $G$ is not uniscalar. Furthermore, if $h\in G$, we have $\omega_+(hgh^{-1}) = h\omega_{+}(g)$. We must have $h\omega_+(g) = \omega_{+}(g)$ as otherwise $s(hgh^{-1})> 1$ and  $hgh^{-1}\not\asymp g$ by Theorem \ref{thm:mainthm}.
	
	Suppose $G$ is not uniscalar and fixes a point $\omega\in \partial \Gamma$. There exists $g\in G$ hyperbolic such that $s(g)>1$. Thus $|\partial G|\ge 1$. It suffices to show that if $g,h\in G$ with $s(g) > 1$ and $s(h)> 1$, then $\omega_+(h) = \omega_+(g)$ as Theorem \ref{thm:mainthm} shows $\partial g = \partial h$. We show that $\omega_+(g) = \omega$. Indeed, we have $g(\omega) = \omega$ since $G$ fixes $\omega$. Theorem \ref{thm:class_hyp_action} shows $\omega\in\{\omega_+(g), \omega_-(g)\}$, but Lemma \ref{lem:scale_to_boundary} shows that $\omega \neq \omega_{-}(g)$.
\end{proof}

\begin{remark}
	In contrast with Corollary \ref{cor:inf_dir}, the group given in Example \ref{ex:2_dir} has two directions. 
\end{remark}
\begin{corollary}\label{cor:inf_dir}
	Suppose $G$ is a hyperbolic t.d.l.c. group. Then $|\partial G|$ is either $0$, $1$ or infinite.
\end{corollary}
\begin{proof}
	We have $|\partial G| = 0$ if and only if $G$ is uniscalar since no element can move towards infinity. If $|\partial G| > 1$, then Theorem \ref{thm:mainthm} gives $g,h\in G$ moving towards infinity such that $\omega_{+}(g)\neq \omega_{+}(h)$.
	If $\omega_{+}(h) = \omega_-(g)$, then the orbit of $\omega_+(h)$ under the action of $G$ is infinite by Lemma \ref{lem:scale_to_boundary}. Alternatively, if $\omega_+(h)\neq \omega_-(g)$, then the orbit of $\omega_+(h)$ under the action of $g$, thus $G$, is infinite since any power of $g$ can only fix $\omega_{\pm}(g)$. In both cases, since $\omega_+(g'h(g')^{-1}) = \gamma'\omega_+(h)$ and $s(g'h(g')^{-1}) = s(h)$ for all $g'\in G$, Theorem \ref{thm:mainthm} shows $\partial G$ is infinite.
\end{proof}

The proof of Corollary \ref{cor:dense_image} depends on \cite[Theorem 2]{Woess93}. There, $L(G)$ denotes the set of all accumulation points of a single orbit $Gv$ in $\partial \Gamma$ where $v\in V(\Gamma)$. Since $G$ acts vertex-transitively on any rough Cayley graph $\Gamma$, we have the simplification $L(G) = \partial \Gamma$. 

\begin{corollary}\label{cor:dense_image}
	Suppose $G$ is a hyperbolic t.d.l.c. group with rough Cayley graph $\Gamma$ such that $|\partial G| > 1$. Then the map given in Theorem \ref{thm:mainthm} has dense image.

\end{corollary}
\begin{proof}

	Since $G$ acts transitively on $\Gamma$, it follows from \cite[Theorem 2 (ii) and Corollary 4]{Woess93} that $H(G) = \{\omega_{+}(g)\mid g\in G \hbox{ is hypebolic}\}$ is dense in $\partial \Gamma$. Thus is $\omega\in \partial \Gamma$ and $U\subset \partial \Gamma$ is open and contains $U$, then there exists $h\in G$ hyperbolic with $\omega_+(h)\in U$. Since $|\partial G|>1$, Theorem \ref{thm:mainthm} gives $g\in G$ hyperbolic and moving towards infinity with $\omega_{+}(g)\neq \omega_+(h)$. Then the sequence $\omega_+(h^ngh^{-n}) = h^n\omega_+(g)$ converges to $\omega_+(h)$ and is therefore eventually contained in $U$. Moreover, $s(h^ngh^{-n}) = s(g)> 1$ and so $h^ngh^{-n}$ moves towards infinity for each $n\in \bN$. Thus the map given by Theorem \ref{thm:mainthm} maps the sequence $\partial (h^ngh^{-n})\in \partial G$ to $h^n(\omega_{+}(g))\in \partial \Gamma$ which is eventually contained in $U$. 
\end{proof}

\bibliography{RefFile}
\bibliographystyle{alpha}
\end{document}